\documentclass[12pt]{article}
\usepackage{mathrsfs}
\usepackage{bbm}

\usepackage{amssymb,a4}
\usepackage{amsmath,amsfonts,amssymb,amsthm}
\usepackage{diagbox}
\usepackage{caption}
\usepackage{graphicx}

\newcommand{\C}{{\mathbb C}}

\newcommand{\fg}{\widehat{\frak{g}}}

\def\<{\langle}
\def\>{\rangle}

\def\g{\mathfrak g}

\newtheorem{theorem}{Theorem}[section]

\newtheorem{lemma}[theorem]{Lemma}

\newtheorem{definition}[theorem]{Definition}

\begin{document}

\begin{center}{Structure and representations of the coset vertex operator algebra  $C( L_{\widehat{osp(1|2)}}(2,0), L_{\widehat{osp(1|2)}}(1,0)^{\otimes 2})$  \footnote{Supported by China NSF grant No.12171312.} }
\end{center}
\begin{center}
	{ Runkang Feng and Cuipo Jiang}\\
	School of Mathematical Sciences, Shanghai Jiao Tong University, Shanghai 200240, China
\end{center}


\begin{abstract} 
	In this paper, we determinate the structure of the coset vertex operator algebra  $C( L_{\widehat{osp(1|2)}}(2,0), L_{\widehat{osp(1|2)}}(1,0)^{\otimes 2})$. We  prove that  $C( L_{\widehat{osp(1|2)}}(2,0), L_{\widehat{osp(1|2)}}(1,0)^{\otimes 2})$ is an extension of the rational vertex operator algebra $L(c_{10,7},0)$. Representations and fusion rules  for   $C( L_{\widehat{osp(1|2)}}(2,0), L_{\widehat{osp(1|2)}}(1,0)^{\otimes 2})$ are also completely determined.
\end{abstract}

\section{Introduction}

Let $V$ be a vertex operator algebra with conformal vector $\omega$. Let $U$ be a vertex operator subalgebra of $V$ with conformal vector $\omega'$. It was shown in \cite{FZ} (also see \cite{LL}) that under mild conditions, the centralizer $C(U,V)$ of $U$ in $V$ is still a vertex operator algebra with central charge $\omega-\omega'$, which is usually called a coset construction. Coset vertex algebras arisen firstly from the Goddard-Kent-Olive construction \cite{GK}  of  the unitary Virasoro vertex  algebras as centralizers of  vertex subalgebras in large ones. Nowadays, coset construction has become one of the major ways to give new vertex algebras from the known ones. 

Another motivation of the present work is about tensor decompositions of affine vertex operator (super)algebras. Let $\g$ be a finite-dimensional simple Lie algebra over ${\mathbb C}$. For a positive integer $k$,  denote by $L_k(\g)$ the simple quotient of $V^k(\g)$.  It is well-known that $L_k(\g)$ is rational  and $C_2$-cofinite \cite{DL2}, \cite{FZ}.
By \cite{HL},  for any positive integer $l$, the tensor product $L_k(\g)^{\otimes l}$ is still a vertex algebra and rational.  As a module of the affine Lie algebra $\fg$, $L_k(\g)^{\otimes l}$ is completely reducible, and has a submodule $L_{lk}(\g)$ generated by ${\bf 1}^{\otimes l}$, which is a vertex operator subalgera of $L_k(\g)^{\otimes l}$.  $L_{lk}(\g)$ is not a conformal vertex subalgebra of $L_k(\g)^{\otimes l}$, and  it is interesting to study the centralizer of $L_{lk}(\g)$ in $L_k(\g)^{\otimes l}$. If $\g=\frak{sl}_n$, the coset vertex algebra $C(L_{lk}(\g), L_k(\g)^{\otimes l})$ turns out to be the parafermion vertex operator algebra $K(\frak{sl}_l, k)$ \cite{JL}, \cite{Lam}, \cite{JL1}  which presents a delicate level-rank duality. When $\g$ is an orthogonal finite-dimensional simple Lie algebra, the centralizer was given in  \cite{JL2} which also presents some kind of level-rank duality. 
If $\g$ is not a Lie algebra, Gorelik and Kac \cite{GK} claimed that $L_{\hat{\g}}(k,0)$ is $C_2$-cofinite if and only if $\g$ is the simple Lie superalgebra $osp(1|2n)$ and $k$ is a positive integer, which was proved recently in \cite{AL} and \cite{CL}. The rationality of $L_{\hat{\g}}(k,0)$ was also proved in  \cite{CL}. 
Since $osp(1|2)$ is the closest one  to $sl_2(\C)$ among the finite-dimensional simple Lie superalgebras, it is natural to study the case that $\g=osp(1|2)$ first, when we consider the supercases. In this paper, we study the coset vertex algebra $C( L_{\widehat{osp(1|2)}}(2,0), L_{\widehat{osp(1|2)}}(1,0)^{\otimes 2})$.
Our goal is to find some kind of building blocks of the coset vertex algebras for general positive integers $l$. We prove that the coset vertex algebra 
$C( L_{\widehat{osp(1|2)}}(2,0), L_{\widehat{osp(1|2)}}(1,0)^{\otimes 2})=L(c_{10,7},0) \oplus L(c_{10,7},h_{10,7}^{6,1})$, where $c_{p,q}=1-\frac{6(p-q)^2}{pq}$, for $p,q\in{\mathbb Z}_{\geq 3}$, $(p,q)=1$, and $h_{p,q}^{r,s}=\frac{(sq-rp)^2-(p-q)^2}{4pq}$, for $0 \leq r \leq q-1$, $0 \leq s \leq p-1$.  This interestingly gives a realization of the conformal extension algebra $L(c_{10,7},0) \oplus L(c_{10,7},h_{10,7}^{6,1})$ of the Virasoro vertex algebra $L(c_{10,7},0)$. By \cite{W}, $L(c_{10,7}, 0)$ is rational and $C_2$-cofinite. Then by \cite{HKL}, $L(c_{10,7},0) \oplus L(c_{10,7},h_{10,7}^{6,1})$ is rational. 
We  further determine all the irreducible modules and fusion rules of the coset vertex algebra.

The paper is organized as follows. In the second section, we give some preliminaries. In Section 3, we determine the structure of the  coset vertex operator algebra $C( L_{\widehat{osp(1|2)}}(2,0), L_{\widehat{osp(1|2)}}(1,0)^{\otimes 2})$. Section 4 is dedicated to representations and fusion rules of $C( L_{\widehat{osp(1|2)}}(2,0), L_{\widehat{osp(1|2)}}(1,0)^{\otimes 2})$. 

\vskip 0.2cm
\noindent{\bf Acknowledgement}  We would like to thank the referee for very careful reading  and invaluable comments and suggestions
\section{Preliminaries}

\subsection{Vertex operator superalgebras}
In this subsection, we review  the definition of vertex operator superalgebra  and various notions of  modules \cite{B}, \cite{D}, \cite{DL}, \cite{DL3}, \cite{FF},  \cite{FLM}, \cite{K}, \cite{Li}, \cite{Z}.
\begin{definition}
	A vertex operator superalgebra $V$ is a 4-tuple $(V, Y(\cdot,z), {\bf 1}, \omega)$ which consists of a $\mathbb{Z}_2$-graded vector space $V=V^{\bar{0}}\oplus V^{\bar{1}}$,  a linear map  	
	\begin{align*}
	Y(\cdot, x): & V \rightarrow(\text { End } V)\left[\left[x, x^{-1}\right]\right] \\ & v \mapsto Y(v, x)=\sum_{n \in \mathbb{Z}} v_{n} x^{-n-1}
	\end{align*}
	and vectors ${\bf 1}, \omega\in V^{\bar{0}}$ 
	such that the following conditions hold  for  $u\in V^{\bar{k}}, v\in V^{\bar{l}}$ and $p,q,r,n\in{\mathbb Z}$, 
	\begin{itemize}
		\item $u_{n}v\in V^{\bar{k}+\bar{l}}$, and  
		$u_{n}v = 0$ for  sufficiently large $n$;
		\item $Y({\mathbf 1}, z)={\rm id}_V$, and 
		$Y(v,x)\mathbf{1} \in V[[x]]$ and $\lim_{x \to 0}Y(v,x)\mathbf{1} = v$;
		\item The Borcherds identity
		\begin{equation}
		\sum\limits_{i=0}^{\infty}\left(\begin{array}{c}q\\i\end{array}\right)
		(u_{p+i}v)_{q+r-i}=	\sum\limits_{i=0}^{\infty}(-1)^i\left(\begin{array}{c}p\\i\end{array}\right)(u_{p+q-i}v_{r+i}-(-1)^{p+kl}
		v_{p+r-i}u_{q+i});
		\end{equation} 
		\item 
		$$
		[L_m, L_n]=(m-n)L_{m+n}	+\delta_{m+n,0}\frac{m^3-m}{12}c_V{\rm id}_{V},
		$$
		for $m,n\in{\mathbb Z}$, where $L_m=\omega_{m+1}$ and $c_V\in{\mathbb C}$ is called the central charge of $V$;
		\item $L_{-1}u=u_{-2}{\bf 1}$;
		\item $V$ is a direct sum of finite-dimensional eigenspaces $V_n$ for $L_0$ of eigenvalues $n\in\frac{1}{2}{\mathbb Z}_{\geq 0}$ as 
		$$
		V=\bigoplus_{n\in\frac{1}{2}{\mathbb Z}_{\geq 0}}V_n.
		$$
	\end{itemize}
\end{definition}
For simplicity, we refer $V$ to be a vertex operator superalgebra, and for $u\in V^{\bar{k}}$, we denote $k$ by $p(u)$.  A vertex operator algebra is a vertex operator superalgebra $V$ satisfying that $V^{\bar{1}}=0$ and $V_{n}=0$ for $n\in\frac{1}{2}+{\mathbb Z}$ \cite{FLM}, \cite{FHL}, \cite{LL}. 

Let $V$ be a vertex operator superalgebra. 
A weak $V$-module is a pair $(M, Y(\cdot, z))$ of a ${\mathbb Z}_2$-graded vector space  
$M=M^{\bar{0}}+M^{\bar{1}}$ and  a linear map	
\begin{align*}
Y_{M}: & V \rightarrow(\text { End } M)\left[\left[z, z^{-1}\right]\right] \\ & v \mapsto Y_{M}(v, z)=\sum_{n \in \mathbb{Z}} v_{n} z^{-n-1}
\end{align*}
satisfying the following conditions for any $v \in V^{\bar{k}}$,$w \in M^{\bar{l}}$:
\begin{itemize}
	\item $v_{n}w\in M^{\bar{k}+\bar{l}}$;
	\item  $v_{n}w = 0$ for a sufficiently large $n \in \mathbb{N}$;
	\item  $Y_{M}(\mathbf{1},z) = id_{M}$;
	\item The Jacobi identity:
	\begin{equation}
	\begin{array}
	{l}{z_{0}^{-1} \delta\left(\frac{z_{1}-z_{2}}{z_{0}}\right) Y_M\left(u, z_{1}\right) Y_M\left(v, z_{2}\right)-(-1)^{p(u)p(v)}z_{0}^{-1} \delta\left(\frac{z_{2}-z_{1}}{-z_{0}}\right) Y_M\left(v, z_{2}\right) Y_M\left(u, z_{1}\right)} \\ {\quad=z_{2}^{-1} \delta\left(\frac{z_{1}-z_{0}}{z_{2}}\right) Y_M\left(Y\left(u, z_{0}\right) v, z_{2}\right).}
	\end{array}
	\end{equation} 
\end{itemize}
Let $V$ be a vertex operator superalgebra satisfying $V_0={\mathbb C}{\bf 1}$ and $V_n=0$ for $n\in\frac{1}{2}+{\mathbb Z}_{\geq 0}$. A weak $V$-module is called  a ${\mathbb Z}_{\geq 0}$-gradable 
weak module if $M$ has a ${\mathbb Z}_{\geq 0}$-grading as $M=\oplus_{n\in{\mathbb Z}_{\geq 0}}M(n)$ and satisfies the condition
$$v_{m}M(n) \in M(n+{\rm wt} (v) -m -1).$$
A vertex operator superalgebra is said to be rational  if any ${\mathbb Z}_{\geq 0}$-gradable  weak $V$-module is completely reducible.

We call a weak $V$-module $M$ a $V$-module if $M$  carries a $\mathbb{C}$-grading $M = \oplus_{t \in \mathbb{C}}M_{t}$ satisfying 
\begin{itemize}
	\item $\mathrm{dim}(M_{t}) < \infty$;
	\item $M_{t+n} = 0$ for a fixed $t \in \mathbb{C}$ and sufficiently large $n \in \mathbb{N}$;
	\item  $ L(0) w=t w=(\mathrm{wt} w) w$ for any  $w \in M_{t}$.
\end{itemize}

Let $(V_{1},Y,\mathbf{1},\omega^{(1)})$, $\cdots,$ $(V_{r},Y,\mathbf{1},\omega^{(r)})$ be vertex superalgebras. The tensor product of vector spaces: 
\begin{align}\label{ver}
V = V_{1} \otimes \cdots \otimes V_{r}
\end{align}
has a vertex operator superalgebra with the linear map defined as
$$
\begin{array}{ll}
& Y(u^{(1)} \otimes \cdots\otimes  u^{(r)},x)(v^{(1)}\otimes\cdots\otimes  v^{(r)})\\
&  =(-1)^{\sum\limits_{i=2}^rp(u^{(i)})(p(v^{(1)})+\cdots p(v^{(i-1)}))}Y(u^{(1)},x)v^{(1)} \otimes \cdots \otimes Y(u^{(r)},x)v^{(r)}.
\end{array}
$$
The conformal vector $\omega$ and the vacuum vector are 
\begin{align*}
\omega=  \omega^{(1)}\otimes \cdots \otimes  \mathbf{1}+\cdots+{\bf 1}\otimes \cdots \otimes  \omega^{(r)}
\end{align*}
and
\begin{align*}
\mathbf{1} =  \mathbf{1} \otimes \cdots \otimes  \mathbf{1}
\end{align*}
respectively.  It is well known that	$V_{1} \otimes \cdots \otimes V_{r}$ is rational if and only if $V_i, 1\leq i\leq r$ are all rational.
\subsection{Fusion rules for vertex operator algebras}
In this subsection we review the notions of intertwining operators and fusion rules from  \cite{DLWY}, \cite{DW3}, \cite{FHL}, \cite{HL2}, \cite{JW2}, \cite{Wang}.
\begin{definition}
	Let $(V,Y,1,\omega)$ be a vertex operator algebra and let $(W^{1},Y^{1}),(W^{2},Y^{2})$ and $(W^{3},Y^{3})$ be $V$-modules. Then an intertwining operator of type $\binom{W^3}{W^1 W^2} $ is a linear map:
	\begin{align*}
	\begin{gathered}
	I(\cdot, z): W^1 \rightarrow \operatorname{Hom}\left(W^2, W^3\right)\{z\} \\
	u \rightarrow I(u, z)=\sum_{n \in \mathbb{Z}} u_n z^{-n-1}
	\end{gathered}
	\end{align*}
	satisfying that
	\begin{itemize}
		\item $u_{n}v = 0$ for $n$ sufficiently large and $u \in W^{1}, v \in W^{2}$;
		\item $I(L(-1)v,z) =\frac{d}{dz}I(v,z)$;
		\item The Jacobi identity:
		\begin{equation}
		\begin{array}
		{l}{z_{0}^{-1} \delta\left(\frac{z_{1}-z_{2}}{z_{0}}\right) Y^{3}\left(u, z_{1}\right) I\left(v, z_{2}\right)-z_{0}^{-1} \delta\left(\frac{z_{2}-z_{1}}{-z_{0}}\right) I\left(v, z_{2}\right) Y^{2}\left(u, z_{1}\right)} \\ {\quad=z_{2}^{-1} \delta\left(\frac{z_{1}-z_{0}}{z_{2}}\right) I\left(Y^{1}\left(u, z_{0}\right) v, z_{2}\right)},
		\end{array}
		\end{equation} 
		where $u \in V$, $v \in W^{1}$.
	\end{itemize}
	
	The space of all intertwining operators of type $\binom{W^3}{W^1 W^2} $ is denoted as
	\begin{align*}
	I_{V}\binom{W^3}{W^1 W^2}.
	\end{align*}
	
	The dimension of the space $I_{V}\binom{W^3}{W^1 W^2}$ is called the fusion rule which is denoted by
	\begin{align*}
	N^{W_{3}}_{W^{1},W^{2}} = \text{dim} I_{V}\binom{W^3}{W^1 W^2}.
	\end{align*}
\end{definition}

\begin{definition}
	Let $V$ be a vertex operator algebra and $W^{1},W^{2}$ be two $V$ modules. Assume $W$ is a $V$-module and $I \in  I_{V}\binom{W}{W^1 W^2}$. If for any $V$-module $M$ and \begin{align*}
	\mathcal{Y} \in  I_{V}\binom{M}{W^1 W^2}.
	\end{align*}
	there is a unique $V$-module homomorphism $f: W \rightarrow M$ such that $\mathcal{Y} = f \circ I$,
	then $(W,I)$ is called a tensor product (or fusion product) which is denoted as $W^{1} \boxtimes_{V} W^{2}$. 
\end{definition}	
The basic result is that the fusion product exists if $V$ is rational and $C_2$-cofinite, and 
\begin{align}
W^1 \boxtimes_V W^2=\sum\limits_{W} N_{W^1, W^2}^W W,
\end{align}
where $W$ runs over the set of equivalence classes of irreducible $V$-modules. 
We now have the following result which was essentially proved in
\cite{ADJR}.

\begin{theorem}\label{n2.2c}
	Let $V^1, V^2$ be rational vertex operator algebras. Let $M^1 , M^2,
	M^3$ be $V^1$-modules and  $N^1, N^2, N^3$ be $V^2$-modules such
	that
	$$dim I_{V^1}\left(\begin{tabular}{c}
	$M^3$\\
	$M^1$ $M^2$\\
	\end{tabular}\right)< \infty , \ dim I_{V^2}\left(\begin{tabular}{c}
	$N^3$\\
	$N^1$ $N^2$\\
	\end{tabular}\right)< \infty.$$
	Then the linear map
	$$\sigma: I_{V^1}\left(\begin{tabular}{c}
	$M^3$\\
	$M^1$ $M^2$\\
	\end{tabular}\right)\otimes I_{V^2}\left(\begin{tabular}{c}
	$N^3$\\
	$N^1$ $N^2$\\
	\end{tabular}\right)\rightarrow I_{V^1\otimes V^2}\left( \begin{tabular}{c}
	$M^3\otimes N^3$\\
	$M^1\otimes N^1$ $M^2\otimes N^2$\\
	\end{tabular}\right)$$
	$$\mathcal{Y}_1( \cdot, z)\otimes \mathcal{Y}_2(\cdot, z)\mapsto (\mathcal{Y}_1\otimes
	\mathcal{Y}_2)(\cdot, z)$$\\is an isomorphism, where
	$(\mathcal{Y}_1\otimes \mathcal{Y}_2)(\cdot, z)$ is defined by
	$$(\mathcal{Y}_1\otimes \mathcal{Y}_2)(\cdot, z)(u^1\otimes v^1,z)(u^2\otimes v^2)= \mathcal{Y}_1(u^1,z)u^2\otimes \mathcal{Y}_2(v^1,z)v^2.$$
\end{theorem}
\subsection{Rational Virasoro vertex operator algebras}

For $c\in{\mathbb C}$, let $L(c,0)$ be the simple Virasoro vertex operator algebra of central charge $c$  \cite{FZ}. Recall from \cite{DMZ}, \cite{W}, $L(c,0)$ is rational if and only if $c=c_{p,q}=1-\frac{6(p-q)^2}{pq}$, $p,q\in{\mathbb Z}_{\geq 3}$, $(p,q)=1$ and 
\begin{align}\label{6}
\{V_{r,s} = L(c_{p,q},h_{p, q}^{r, s})|h_{p, q}^{r, s} \in S_{p,q}\}
\end{align}
exhaust all irreducible $L(c_{p,q},0)$-modules, where
\begin{equation}\label{hrs}
\begin{aligned}
S_{p,q}=\left\{\left.h_{p, q}^{r, s}=\frac{(s q-r p)^2-(p-q)^2}{4 p q} \right\rvert\, 1 \leq r \leq q-1,1 \leq s \leq p-1\right\}.
\end{aligned}
\end{equation}
We denote the irreducible modules $L(c_{p,q},h_{p, q}^{r, s})$ by $V_{r,s}$ for simplicity,  where $ 1 \leq r \leq q-1,1 \leq s \leq p-1$.

An ordered triple of pairs of integers $((r,s), (r',s'), (r'',s''))$ is called {\em admissible}, if $1\leq r,r',r''\leq q-1$, $1\leq s,s',s''\leq p-1$, $ r+r'+r''\leq 2q-1$, $s+s'+s''\leq 2p-1$, $r<r'+r'', r'<r+r'', r''<r+r', s<s'+s'', s'<s+s'', s''<s+s'$, and both $r+r'+r''$ and $s+s'+s''$ are odd. The following result comes from  \cite{W} ( for $c=\frac{1}{2}$, also see \cite{DMZ}).
\begin{theorem}\label{wang}\cite{W}
	Denote
	$$N_{(r,s),(r',s')}^{(r'',s'')}=dim I_{L(c_{p,q},0)}\left(\begin{tabular}{c}
	$V_{r'',s''}$\\
	$V_{r,s}$ $V_{r',s'}$\\
	\end{tabular}\right).$$
	Then $N_{(r,s),(r',s')}^{(r'',s'')}=1$  if and only if $((r,s), (r',s'), (r'',s''))$ is admissible; otherwise, $N_{(r,s),(r',s')}^{(r'',s'')}$ $=0$. 
\end{theorem}
\subsection{Affine vertex superalgebras}

In this subsection, we will review affine vertex superalgebras following \cite{KRW} and \cite{FZ}.
Let $\mathfrak{g}$ be a finite-dimensional Lie superalgebra equipped with a non-degenerate even supersymmetric bilinear form $(\cdot|\cdot)$ which is normalized by
\begin{align*} 
(\cdot|\cdot)=\frac{1}{2 h^{\vee}} \times \text { Killing form of } \mathfrak{g}.
\end{align*}

Let $\widehat{\mathfrak{g}}=\mathfrak{g}\otimes\mathbb{C}[t,t^{-1}]\oplus\mathbb{C} \mathbf{k}$
be the  affine Lie superalgebra associated to $\g$ with the following super Lie bracket relation:
\begin{align}
[a \otimes t^{m}, b \otimes t^{n}]=[a,b] \otimes t^{m+n} + m(a|b)\delta_{m+n,0}\mathbf{k},
\end{align}
where $a,b \in \mathfrak{g}$, $m,n \in \mathbb{Z}$, and 
$\mathbf{k}$ is the central element of $\widehat{\mathfrak{g}}$.
The affine Lie superalgebra $\widehat{\mathfrak{g}}$ has the following triangular decomposition:
\begin{align*}
\widehat{\mathfrak{g}}=\widehat{\mathfrak{g}}_{(0)}\oplus \widehat{\mathfrak{g}}_{(+)}\oplus \widehat{\mathfrak{g}}_{(-)},
\end{align*}
where
\begin{align*}
\widehat{\mathfrak{g}}_{(0)} = \mathfrak{g} \oplus \mathbb{C} \mathbf{k},  \quad
\widehat{\mathfrak{g}}_{(\pm)}=\coprod_{n \in \mathbb{Z}_{+}} \widehat{\mathfrak{g}}_{(\pm n)},  \quad 
\widehat{\mathfrak{g}}_{(n)} =\mathfrak{g} \otimes t^{n}. 
\end{align*}
For $l\in{\mathbb C}$, let ${\mathbb C}_{l}={\mathbb C}$ be the  $\widehat{\mathfrak{g}}_{(0)}\oplus \widehat{\mathfrak{g}}_{(+)}$-module such that 
$\widehat{\mathfrak{g}}_{+}$ and $\mathfrak{g}$ acts trivially on $\mathbb{C}_{l}$ and $\mathbf{k}$ acts as the scalar $l$. Then we have the following Verma module 
\begin{align*}
V_{\widehat{\mathfrak{g}}}(l,0)=U(\widehat{\mathfrak{g}}) \otimes_{U(\widehat{\mathfrak{g}}_{(0)} \oplus \widehat{\mathfrak{g}}_{(+)})} \mathbb{C}_{l}.
\end{align*}
For $a\in\g$ and $n\in{\mathbb Z}$, we denote $a\otimes t^n$ by $a(n)$. 
The space $V_{\widehat{\mathfrak{g}}}(l,0)$ (also denoted by $V^l(\frak{g})$ in some references) is naturally graded as
\begin{align*}
V_{\widehat{\mathfrak{g}}}(l,0) = \coprod_{n \ge 0}  V_{\widehat{\mathfrak{g}}}(l,0)_{(n)},
\end{align*}
where $V_{\widehat{\mathfrak{g}}}(l,0)_{(n)}$ is linearly spanned by the vectors
\begin{align*}
a^{(1)}(-m_{1})a^{(2)}(-m_{2})\cdots a^{(r)}(-m_{r})\mathbf{1}
\end{align*}
for $r \ge 0,a^{(i)} \in \mathfrak{g}, m_{i} \ge 1$ with $n=m_1+\cdots+m_r$. There exists a unique vertex superalgebra structure $(V_{\widehat{\mathfrak{g}}}(l,0),Y,\mathbf{1})$ on $V_{{\widehat{\mathfrak{g}}}}(l,0)$
such that $\mathbf{1}=1 \in \mathbb{C}$ is the vacuum vector and the generating function of $a(-1){\bf 1}$ is
\begin{align*}
Y(a(-1){\bf  1},x)=a(x) = \sum_{n \in \mathbb{Z}}a(n)x^{-n-1} \in \widehat{\mathfrak{g}}[[x,x^{-1}]], 
\end{align*}
for $a\in\g$.  The vertex superalgebra $(V_{\widehat{\mathfrak{g}}}(l,0),Y,\mathbf{1})$ is called the affine vertex  superalgebra associated to  $\mathfrak{g}$ at level $l$.
We identify $\mathfrak{g}$ with $ V_{\widehat{\mathfrak{g}}}(l,0)_{(1)}$ via the linear isomorphism:
\begin{align*}
& \mathfrak{g} \rightarrow V_{\widehat{\mathfrak{g}}}(l,0) \\ & a \mapsto a(-1)\mathbf{1}.
\end{align*}

Let $\{a_i, 1\leq i\leq d \} $ and $\{a^i , 1\leq i\leq d\}$ be a pair of dual bases of $\frak{g}$ such that $(a_i |a^j ) = \delta_{ij}$. Then $\Omega=\sum\limits_{i=1}^d(-1)^{p(a_i)p(a^i)}a_ia^i$ is the Casimir operator of $\frak{g}$, which lies in the center of $U(\frak{g})$. The dual Coxeter number $h^{\vee}$ of $\frak{g}$ is defined as the one-half of the eigenvalue of $\Omega$ in the adjoint representation. By the Sugawara construcion, if $l+h^{\vee}\neq 0$, $V_{\widehat{\mathfrak{g}}}(l,0)$ is a vertex operator superalgebra with the conformal vector:
\begin{align}
\omega=\frac{1}{2(l+h^{\vee})} \sum_{i=1}^d (-1)^{p(a_i)p(a^i)}a_{(i)}(-1) a^{(i)}(-1){\bf 1}
\end{align}
of  central charge $c = \frac{l\cdot{\rm dim}\frak{g}}{l+h^{\vee}} $ \cite{KRW}.

\vskip 0.2cm
Denote by $L_{\widehat{\mathfrak{g}}}(l,0)$ (also denoted by $L_l(\frak{g})$ in some references) the simple quotient of $V_{\widehat{\mathfrak{g}}}(l,0)$. 
$L_{\widehat{\mathfrak{g}}}(l,0)$ is called the simple affine vertex superalgebra associated to  $\mathfrak{g}$ at level $l$.

\vskip 0.2cm
We now assume that $\frak{g}=osp(1|2)$ \cite{Ksuper}. 
Let $\{e,f,h,x,y\}$ be the basis of the Lie superalgebra $osp(1|2)$ with the anti-commutation relations:
\begin{align*}
&[e,f]= h,[h,e] = 2e,[h,f] = -2f, \\
& [h,x]= x,[e,x] = 0, [f,x] = -y, \\
& [h,y] = -y,[e,y] = -x,[f,y] = 0, \\
& \{x,x\} =2e,\{x,y\} = h, \{y,y\} = -2f.
\end{align*}
Let  $(\cdot \mid \cdot)$ be the  invariant even supersymmetric nondegenerate bilinear form on $osp(1|2)$ such that $(\alpha \mid \alpha) =2$ if $\alpha$ is a long root in the root system of even part where we identify the Cartan subalgebra $\mathfrak{h}$ with $\mathfrak{h}^{\vee}$ via the bilinear form $(\cdot \mid \cdot)$. Then the conformal vector $\omega$ of $V_{\widehat{\mathfrak{g}}}(l,0)$ is given as 
\begin{align*}
\omega & = \frac{1}{2(l+\frac{3}{2})}(\frac{1}{2}h(-1)h(-1)\mathbf{1}+ e(-1)f(-1)\mathbf{1} + f(-1)e(-1)\mathbf{1} \\
& - \frac{1}{2}x(-1)y(-1)\mathbf{1} + \frac{1}{2}y(-1)x(-1)\mathbf{1} )
\end{align*}
of central charge $\frac{2l}{2l+3}$ \cite{KRW}. If $l\in{\mathbb Z}_{\geq 0}$,  the simple affine vertex operator superalgebra $L_{\widehat{osp(1|2)}}(l,0)$ is rational and $C_2$-cofinite \cite{AL}, \cite{CL}.

\section{Structure of the VOA
	$C( L_{\widehat{osp(1|2)}}(2,0), L_{\widehat{osp(1|2)}}(1,0) ^{\otimes 2})$}

\subsection{Commutant vertex operator superalgebras}
In this subsection, we review commutant vertex operator (super)algebras following \cite{CRS}, \cite{CLi}, \cite{FZ},  \cite{LL}, although  these references are about vertex operator algebras.
Let $V$ be a vertex operator superalgebra and $U$  a vertex operator super-subalgebra of $V$. Set
\begin{align*}
C(U,V) = \{v \in V|[Y(u,z_{1}),Y(v,z_{2})] = 0, u \in U   \},
\end{align*}
where 
$$
[Y(u,z_1), Y(v,z_2)]=Y(u,z_1)Y(v,z_2)-(-1)^{p(u)p(v)}Y(v,z_1)Y(u,z_1).
$$
By the Jacobi identity (see also \cite{LL}),  we have
\begin{align*}
C(U,V) = \{v \in V|u_{m}v = 0, u \in U, m \geq 0   \}.
\end{align*}
The following result is about vertex operator algebras.  But it is also true for vertex operator superalgebras.
\begin{theorem}\cite{FZ}, \cite{LL}
	Let $(V,Y,1,\omega)$ be a nonzero vertex operator algebra such that $V=\oplus_{n=0}^{\infty}V_{n}$ and $V_{0} = \mathbb{C}\mathbf{1}$. Let  $(U,Y,1,\omega')$ be a vertex operator subalgebra of $V$ and assume that
	\begin{align*}
	\omega' \in U \cap V_{(2)}
	\end{align*}
	satisfying
	\begin{align*}
	L(1)\omega' = 0.
	\end{align*}
	Set
	\begin{align*}
	Y(\omega', x)=\sum_{n \in \mathbb{Z}} L'(n) x^{-n-2}
	\end{align*}
	acting on all elements of $V$. 
	
 Then the gradings of $V$ and $U$ are compatible ( i.e. $L(0) = L'(0)$ on $U$) and more generally,
	\begin{align*}
	L(n) = L'(n) \quad \text{on} \quad U \quad  \text{for all} \quad n \geq -1.
	\end{align*}
   Set $\omega'' = \omega - \omega'$, then $\omega'' \in C(U,V)$ and $(C(U,V),Y,1,\omega'')$ is a vertex operator algebra of central charge equal to $c_{V} - c_{U}$ and 
	\begin{align*}
	L(1)\omega'' = 0.
	\end{align*}
\end{theorem}
 Let $l_1, \cdots, l_s$ be positive integers. For the simple vertex operator superalgebras $ L_{\widehat{osp(1|2)}}(l_i,0)$ associated to ${osp(1|2)}$ at levels $l_{i}$,  $1\leq i\leq s$,  we have  the tensor product vertex operator superalgebra
\begin{align*}
    V=  L_{\widehat{osp(1|2)}}(l_{1},0) \otimes  L_{\widehat{osp(1|2)}}(l_{2},0) \otimes \cdots \otimes  L_{\widehat{osp(1|2)}}(l_{s},0).
\end{align*}
 Let $l = l_{1} + l_{2} + \cdots + l_{s}$.  Since  the Lie superalgebra $osp(1|2)$ can be naturally embedded into the weight one subspace of $V$ by
\begin{align*}
      a \mapsto a(-1)\mathbf{1} \otimes \mathbf{1} \otimes \cdots \otimes \mathbf{1} + \mathbf{1} \otimes a(-1)\mathbf{1} \otimes \cdots \otimes \mathbf{1} + \cdots + \mathbf{1} \otimes \mathbf{1} \otimes \cdots \otimes a(-1) \mathbf{1},
\end{align*}
$V$ has a   simple vertex operator superalgebra $U$ isomorphic to  $ L_{\widehat{osp(1|2)}}(l,0)$. Then the commutant  $C(U,V)$ of $U$ in $V$ is a vertex operator (super)subalgebra of $V$. We will focus on the case that $s=2$, and  $l_1=l_2=1$. 
\subsection{Characters of irreducible modules of  $L_{\widehat{osp(1|2)}}(l,0)$}

 For a positive integer $l$, let  $L_{\widehat{sl_{2}}}(l, 0)$  be the simple affine vertex operator algebra associated to $sl_2$ at level $l$.  It is well-known that  $L_{\widehat{sl_{2}}}(l, 0)$ has $l+1$ irreducible modules $L_{\widehat{sl_{2}}}(l, i)$, $0\leq i\leq l$
 \cite{FZ}, \cite{LL}. 
 Following \cite{CFK} we  denote by $V_{r,s}$ the  irreducible module $L_{\text{Vir}}(c_{2l+3,l+2} ,h_{2l+3, l+2}^{r,s})$ for the rational Virasoro vertex operator algebra $L(c_{2l+3,l+2},0)$. 
   We have the following result from  \cite{CFK}.
\begin{theorem}\label{thmCFK} \cite{CFK} Let $l$ be a positive integer. Then 
all simple modules of the vertex operator superalgebra $L_{\widehat{osp(1|2)}}(l, 0)$ up to isomorphism are given as
$$M_{r}=M_r^{even}\oplus M_r^{odd},  \ 1\leq r\leq 2l+1, \  r\in 2{\mathbb Z}+1,$$
where  $M_1=L_{\widehat{osp(1|2)}}(l,0) $ and 
\begin{align}\label{Meven}
 & M^{{even }}_{r} = \bigoplus_{i=0, { even }}^{l} L(l, i) \otimes V_{i+1,r}, 
 \end{align}
 \begin{align}\label{Modd}
 & M^{{odd }}_{r} = \bigoplus_{i=0, { odd }}^{l} L(l, i) \otimes V_{i+1,r}. 
\end{align}
\end{theorem}
 From \cite{CFK}, \cite{KW}, the characters of these modules are given by
\begin{equation}\label{5}
\begin{aligned}
  \operatorname{ch}\left[M_{r}\right]=\frac{\Theta_{b_{+}, a}\left(\frac{z}{2}, \frac{\tau}{2}\right)-\Theta_{b_{-}, a}\left(\frac{z}{2}, \frac{\tau}{2}\right)}{\Pi(z, t)}
\end{aligned}
\end{equation}
with the Jacobi theta functions
\begin{align*}
\Theta_{b_{\pm}, a}(z, \tau)=\sum_{m \in \mathbb{Z}} w^{a\left(m+\frac{b_{\pm}}{2 a}\right)} q^{a\left(m+\frac{b_{\pm}}{2 a}\right)^2}
\end{align*}
and the Weyl super denominator
\begin{align*}
\Pi(z, \tau)=\Theta_{1,3}\left(\frac{z}{2}, \frac{\tau}{2}\right)-\Theta_{-1,3}\left(\frac{z}{2}, \frac{\tau}{2}\right)=w^{\frac{1}{4}} q^{\frac{1}{24}} \prod_{n=1}^{\infty} \frac{\left(1-q^n\right)\left(1-w q^n\right)\left(1-w^{-1} q^{n-1}\right)}{\left(1+w^{\frac{1}{2}} q^n\right)\left(1+w^{-\frac{1}{2}} q^{n-1}\right)}
\end{align*}
with
\begin{align*}
 w = e^{2\pi iz}, \quad q = e^{2\pi i\tau}, \quad b_{\pm} = \pm r, \quad a =2l+3.
\end{align*}
 The characters are all analytic functions in the domain $1 \leq |w| \leq |q^{-1}|$
which are meromorphically continued to the mermorphic Jacobi forms.

\subsection{The structure of  $C( L_{\widehat{osp(1|2)}}(2,0),L_{\widehat{osp(1|2)}}(1,0) ^{\otimes 2})$}

 Recall (\ref{Meven}) and (\ref{Modd}), we have 
\begin{align*}
L_{\widehat{osp(1|2)}}(l, 0) = L_{\widehat{osp(1|2)}}^{\text {even }}(l, 0) \oplus L_{\widehat{osp(1|2)}}^{\text {odd }}(l, 0)=\bigoplus_{i=0}^{l} L(l, i) \otimes V_{i+1,1}.
\end{align*}
We  have the following lemma.
\begin{lemma}\label{cc}
The conformal vector of  the coset construction  $C( L_{\widehat{osp(1|2)}}(2,0), L_{\widehat{osp(1|2)}}(1,0) ^{\otimes 2})$ is a Virasoro vector with central charge  $c_{10,7}=\frac{8}{35}$.
\end{lemma}

\begin{proof}

Since the central charge of $L_{\widehat{osp(1|2)}}(l, 0)$ is $\frac{2l}{2l+3}$ \cite{KRW},  it follows  that $c_{L_{\widehat{osp(1|2)}}(1, 0)} = \frac{2}{5}$, and $c_{L_{\widehat{osp(1|2)}}(2, 0)} = \frac{4}{7}$. Then we have
\begin{equation}\label{3}
\begin{aligned}
    c_{L_{\widehat{osp(1|2)}}(1, 0)^{\otimes 2}} = 2c_{L_{\widehat{osp(1|2)}}(1, 0)} =   c_{L_{\widehat{osp(1|2)}}(2, 0)} + c_{10,7}.
\end{aligned}
\end{equation}
 This deduces that the central charge of the commutant vertex operator (super)algebra  $C( L_{\widehat{osp(1|2)}}(2,0),L_{\widehat{osp(1|2)}}(1,0) ^{\otimes 2})$ is equal to $c_{10,7}$. 
\end{proof}

We now introduce  the characters of irreducible modules of  the rational  Virasoro vertex operator algebra $L(c_{10,7},0)$ from \cite{An}, \cite{IK}.
\begin{theorem}\label{the1}\cite{An} \cite{IK}
The character of the irreducible $L(c_{10,7},0)$-module $L_{\text{Vir}}(c_{10,7},h_{10,7}^{r,s})$ is given by
\begin{align}
  ch[L_{\text{Vir}}(c_{10,7},h_{10,7}^{r,s})] = [\Theta_{10r-7s, 70}\left(0, \tau\right) - \Theta_{10r+7s, 70}\left(0, \tau\right)]\eta(\tau)^{-1},
\end{align}
where $\eta(\tau) = q^{\frac{1}{24}}\prod_{n=1}^{\infty}(1-q^{n})$ with $q = e^{2\pi i \tau}$.
\end{theorem}
We will also need the following result from \cite{CFK}, \cite{CR2} and  \cite{KW}.
\begin{theorem}\cite{CFK}\cite{CR2} \cite{KW}\label{cha}
	Let $l + \frac{3}{2} = \frac{p}{2p'}$ where $p > 1$ and $p'$ are positive coprime integers with $p+p' \in 2\mathbb{Z}$ and $p, \frac{p+p'}{2}$ are coprime. Let $\Delta = p+p'$, the  characters of $L(l,i)$ can be shown as
	\begin{align}
		\operatorname{ch}\left[L(l,i)\right](z, \tau)=\frac{\left(\Theta_{2 p^{\prime} (i+1), \Delta p^{\prime}}-\Theta_{-2 p^{\prime} (i+1), \Delta p^{\prime}}\right)\left(\frac{z}{2 p^{\prime}}, \frac{\tau}{2}\right)}{\mathfrak{i} \vartheta_1(w, q)}
	\end{align}
	with 
	\begin{align*}
		w = e^{2\pi iz}, \quad q = e^{2\pi i\tau}
	\end{align*}
	and
	\begin{align*}
		\vartheta_2(z, \tau)=\sum_{n \in \mathbb{Z}} w^{\left(n+\frac{1}{2}\right)} q^{\frac{1}{2}\left(n+\frac{1}{2}\right)^2}, \quad \vartheta_1(z, \tau)=-\vartheta_2\left(z+\frac{1}{2}, \tau\right).
	\end{align*}
\end{theorem}
\begin{lemma}\label{lem 2}
\begin{equation}\label{4} 
\begin{aligned}
 & ch_{q}[L_{\widehat{osp(1|2)}}(1, 0) ^{\otimes 2}]   =  q^{-\frac{1}{30}}(1 + 10q + 43q^2 +132q^3+375q^4+946q^5 +2199q^6 \\& + 4852q^7+10188q^8+20542q^9+40084q^{10}+75940q^{11}+ 140219q^{12}+ 253150q^{13}\\& + 447857q^{14}+ 777940q^{15} + 1329196q^{16} + 2236966q^{17}+3712731q^{18}+6083848q^{19}+\cdots),
 \end{aligned}
 \end{equation}
 \begin{equation}\label{4e}
 \begin{aligned}
  & ch_{q}[L_{\widehat{osp(1|2)}}(2, 0)](ch_{q}[L_{\text{Vir}}(c_{10,7},0)]+ch_{q}[L_{\text{Vir}}(c_{10,7},h_{10,7}^{6,1})])\\ &+ ch_{q}[M_{3}] (ch_{q}[L_{\text{Vir}}(c_{10,7},h_{10,7}^{3,1})] + ch_{q}[L_{\text{Vir}}(c_{10,7},h_{10,7}^{4,1})]) \\ &+ ch_{q}[M_{5}](ch_{q}[L_{\text{Vir}}(c_{10,7},h_{10,7}^{2,1})]+ ch_{q}[L_{\text{Vir}}(c_{10,7},h_{10,7}^{5,1})]) \\
  & = q^{-\frac{1}{30}}(1 + 10q + 43q^2 +132q^3+375q^4+946q^5 +2199q^6  + 4852q^7+10188q^8 \\ & +20542q^9+40084q^{10}+75940q^{11}+ 140219q^{12}+ 253150q^{13} + 447857q^{14}+ 777940q^{15}\\& + 1329196q^{16} + 2236966q^{17}+3712731q^{18}+6083848q^{19}+\cdots).
\end{aligned}
\end{equation}
where $M_{3}$ and $M_{5}$ are irreducible modules of $L_{\widehat{osp(1|2)}}(l, 0)$ given in Theorem \ref{thmCFK} with $l = 2$. 
\end{lemma}

\begin{proof} 
According to (\ref{hrs}), all conformal weights of the different irreducible modules of $L_{\text{Vir}}(c_{10,7},0)$ are listed in the Table 1 below, where $1 \leq r \leq 6$ and $1 \leq s \leq 9$.
\begin{table}[h]\label{tabb111}
\centering
\caption{Conformal weights $h_{7,10}^{r,s}$ of  irreducible modules  $L_{\text{Vir}}(c_{10,7},h_{10,7}^{r,s})$ }
\begin{tabular}{|l| c|  c| c| c| c| c|c| c| c|r|}
\hline
\diagbox{r}{s}  &1  &2  & 3  & 4  & 5 & 6 & 7 & 8 & 9  \\ \hline
1  & 0  & 1/40 & 2/5 & 9/8  & 11/5 & 29/8 & 27/5 &301/40 &10 \\ \hline
2  & 4/7  & 27/280 & -1/35  & 11/56  &27/35 &95/56 &104/35 &1287/280 &46/7 \\ \hline
3 & 13/7 & 247/280  & 9/35  & -1/56  & 2/35  & 27/56  & 44/35 & 667/280 & 27/7 \\ \hline
4 &27/7 &667/280 &44/35 &27/56 &2/35 &-1/56 & 9/35 &247/280 &13/7\\ \hline
5 & 46/7 &1287/280 & 104/35 &95/56 & 27/35 & 11/56 & -1/35 &27/280 &4/7\\ \hline
6 & 10 &301/40 &27/5 &29/8 & 11/5 & 9/8 & 2/5 &1/40 &0 \\ \hline
\end{tabular}
\end{table}
 From (\ref{hrs}), we have
\begin{align*}
  h_{10,7}^{r,s} = h_{10,7}^{7-r,10-s}, 1 \leq r \leq 6,1 \leq s \leq 9.
\end{align*}
By  Table 1, the only integral weights of $L_{\text{Vir}}(c_{10,7},0)$  are $ h_{10,7}^{1,1}$ and $h_{10,7}^{6,1}$. 

 Using  (\ref{5}) and Theorem \ref{the1}, we have the following $q$-characters:
 \begin{align*}
  & ch_{q}[L_{\widehat{osp(1|2)}}(1, 0)] = \frac{(\sum_{m \in \mathbb{Z}}(10m+1)q^{\frac{5}{2}(m+\frac{1}{10})^2)})(\Pi^{\infty}_{m=1}(1+q^m)(1+q^m))}{q^{\frac{1}{24}}\Pi^{\infty}_{m=1}(1-q^m)(1-q^m)(1-q^m)}, \\
  &  ch_{q}[L_{\widehat{osp(1|2)}}(2, 0)] = \frac{(\sum_{m \in \mathbb{Z}}(14m+1)q^{\frac{7}{2}(m+\frac{1}{14})^2})(\Pi^{\infty}_{m=1}(1+q^m)(1+q^m))}{q^{\frac{1}{24}}\Pi^{\infty}_{m=1}(1-q^m)(1-q^m)(1-q^m)}, \\
  & ch_{q}[M_{3}] = \frac{(\sum_{m \in \mathbb{Z}}(14m+3)q^{\frac{7}{2}(m+\frac{3}{14})^2})(\Pi^{\infty}_{m=1}(1+q^m)(1+q^m))}{q^{\frac{1}{24}}\Pi^{\infty}_{m=1}(1-q^m)(1-q^m)(1-q^m)}, \\
  &  ch_{q}[M_{5}] = \frac{(\sum_{m \in \mathbb{Z}}(14m+5)q^{\frac{7}{2}(m+\frac{5}{14})^2})(\Pi^{\infty}_{m=1}(1+q^m)(1+q^m))}{q^{\frac{1}{24}}\Pi^{\infty}_{m=1}(1-q^m)(1-q^m)(1-q^m)}, \\
  &  ch_{q}[L_{\text{Vir}}(c_{10,7},h_{10,7}^{r,1})] = \frac{\sum_{m \in \mathbb{Z}}(q^{70(m+\frac{10r-7}{140})^2} -  q^{70(m+\frac{10r+7}{140})^2})}{q^{\frac{1}{24}}\Pi^{\infty}_{m=1}(1-q^m)}
 \end{align*}
with $1 \leq r \leq 6$. Then direct calculation yields that 
\begin{align*}
 & ch_{q}[L_{\widehat{osp(1|2)}}(2, 0)]ch_{q}[L_{\text{Vir}}(c_{10,7},0)]  =  q^{-\frac{1}{30}}(1 + 5q + 19q^2 +48q^3+124q^4+284q^5 \\  &+613q^6 + 1266q^7+2513q^8+4806q^9+8959q^{10}+ 16267q^{11}+ 28895q^{12}+ 50326q^{13}\\& + 86128q^{14}+ 145015q^{15} + 240682q^{16} + 394109q^{17}+637435q^{18}+1019306q^{19}\cdots), \\
 & ch_{q}[L_{\widehat{osp(1|2)}}(2, 0)]ch_{q}[L_{\text{Vir}}(c_{10,7},h_{10,7}^{6,1})] =  q^{-\frac{1}{30}}(q^{10} + 6q^{11}+ 25q^{12}+ 73q^{13} + 197q^{14}\\&+ 481q^{15} + 1093q^{16} + 2354q^{17}+4848q^{18}+9605q^{19}+\cdots), \\ 
 & ch_{q}[M_{3}]ch_{q}[L_{\text{Vir}}(c_{10,7},h_{10,7}^{3,1})]  =  q^{-\frac{1}{30}}( 3q^2 +18q^3+64q^4+189q^5+487q^6 + 1155q^7 \\ & +2561q^8+5394q^9+10879q^{10}+ 21177q^{11}+ 39981q^{12}+ 73510q^{13} + 132042q^{14}\\&+ 232330q^{15} + 401262q^{16} + 681527q^{17}+1140021q^{18}+1880571q^{19} +\cdots), \\
 & ch_{q}[M_{3}]ch_{q}[L_{\text{Vir}}(c_{10,7},h_{10,7}^{4,1})]  =  q^{-\frac{1}{30}}(3q^4+18q^5 +64q^6 +192q^7 + 502q^8+1201q^9 \\ & +2689q^{10}  + 5707q^{11}+ 11593q^{12}+ 22711q^{13} + 43127q^{14}+ 79709q^{15} + 143874q^{16}\\& + 254280q^{17}+440990q^{18}+751891q^{19}+\cdots), \\
 & ch_{q}[M_{5}]ch_{q}[L_{\text{Vir}}(c_{10,7},h_{10,7}^{2,1})]   =  q^{-\frac{1}{30}}(5q+ 21q^2 +66q^3+184q^4+455q^5 +1035q^6 \\ & + 2234q^7+4591q^8+9070q^9+17351q^{10}+ 32257q^{11}+ 58490q^{12}+ 103791q^{13}\\&  + 180603q^{14}+ 308780q^{15} +519629q^{16}+861849q^{17}+1410525q^{18}+2280428q^{19}+\cdots), \\
 & ch_{q}[M_{5}]ch_{q}[L_{\text{Vir}}(c_{10,7},h_{10,7}^{5,1})]   =  q^{-\frac{1}{30}}( 5q^7+21q^8+71q^9+205q^{10}+ 526q^{11}+ 1235q^{12}\\&+ 2739q^{13} + 5760q^{14}+ 11625q^{15} + 22656q^{16}+42847q^{17}+78912q^{18}+142047q^{19}+\cdots), \\
 & ch_{q}[L_{\widehat{osp(1|2)}}(1, 0) ^{\otimes 2}]   =  q^{-\frac{1}{30}}(1 + 10q + 43q^2 +132q^3+375q^4+946q^5 +2199q^6 \\& + 4852q^7+10188q^8+20542q^9+40084q^{10}+ 75940q^{11}+ 140219q^{12}+ 253150q^{13}\\& + 447857q^{14}+ 777940q^{15} + 1329196q^{16} + 2236966q^{17}+3712731q^{18}+6083848q^{19} \cdots).
\end{align*}
We thus obtain  the table below.
\begin{table}[htbp]\label{tab}
\centering
\caption{Coefficients of each power of $q$ for these characters}
\resizebox{\textwidth}{!}{
\begin{tabular}{|l| c| c| c| c| c| c| c| c|c| c| c|r|}
\hline
 & 0 &1  &2  & 3  & 4  & 5 & 6 & 7 & 8 & 9 & 10 \\ \hline
$ch_{q}[L_{\widehat{osp(1|2)}}(2, 0)]ch_{q}[L_{\text{Vir}}(c_{10,7},0)]$& 1  & 5  & 19  & 48  & 124  & 284 & 613 & 1266 &2513 &4806 &8959 \\ \hline
$ch_{q}[L_{\widehat{osp(1|2)}}(2, 0)]ch_{q}[L_{\text{Vir}}(c_{10,7},h_{10,7}^{6,1})]$& 0  & 0  & 0  & 0  &0  &0 &0 &0 &0 &0 &1 \\ \hline
$ ch_{q}[M_{3}]ch_{q}[L_{\text{Vir}}(c_{10,7},h_{10,7}^{3,1})]$& 0 & 0& 3  & 18  & 64  & 189  & 487  & 1155 & 2561 & 5394 &10879  \\ \hline
$ ch_{q}[M_{3}]ch_{q}[L_{\text{Vir}}(c_{10,7},h_{10,7}^{4,1})]$& 0 & 0& 0  & 0  & 3  & 18  & 64  & 192 & 502 & 1201 &2689 \\ \hline
$ch_{q}[M_{5}]ch_{q}[L_{\text{Vir}}(c_{10,7},h_{10,7}^{2,1})]$& 0 & 5 &21 &66 &184 &455 &1035 & 2234 &4591 &9070 &17351\\ \hline
$ch_{q}[M_{5}]ch_{q}[L_{\text{Vir}}(c_{10,7},h_{10,7}^{5,1})]$& 0 & 0 &0 &0 &0 &0 &0 & 5 &21&71& 205\\ \hline
$ ch_{q}[L_{\widehat{osp(1|2)}}(1, 0) ^{\otimes 2}] $& 1 &10 &43 & 132 &375 & 946 & 2199 & 4852 & 10188 &20542 & 40084\\ \hline

Coefficients of items in (\ref{4e})& 1 &10 &43 & 132 &375 & 946 & 2199 & 4852 & 10188 &20542 & 40084\\ \hline
\end{tabular}
}
\end{table}

\newpage 

\begin{table}[h]\label{tab}
\centering
\captionsetup{labelformat=empty} 
\caption{Coefficients of each power of $q$ for these characters(continued)}
\resizebox{\textwidth}{!}{
\begin{tabular}{|l| c| c| c| c| c| c| c| c|c| c| c|r|}
\hline
  &11  &12  & 13  & 14  & 15 & 16 & 17 & 18 & 19 \\ \hline
$ch_{q}[L_{\widehat{osp(1|2)}}(2, 0)]ch_{q}[L_{\text{Vir}}(c_{10,7},0)]$& 16267 & 28895  & 50326  & 86128  & 145015  & 240682 & 394109 & 637435 &1019306 \\ \hline
$ch_{q}[L_{\widehat{osp(1|2)}}(2, 0)]ch_{q}[L_{\text{Vir}}(c_{10,7},h_{10,7}^{6,1})]$& 6  & 25  & 73  & 197  &481  &1093 &2354 &4848 &9605 \\ \hline
$ ch_{q}[M_{3}]ch_{q}[L_{\text{Vir}}(c_{10,7},h_{10,7}^{3,1})]$& 21177 & 39981& 73510  & 132042  & 232330  & 401262  & 681527  & 1140021 & 1880571  \\ \hline
$ ch_{q}[M_{3}]ch_{q}[L_{\text{Vir}}(c_{10,7},h_{10,7}^{4,1})]$& 5707 & 11593& 22711  & 43127  & 79709  & 143874 & 254280 & 440990 &751891 \\ \hline
$ch_{q}[M_{5}]ch_{q}[L_{\text{Vir}}(c_{10,7},h_{10,7}^{2,1})]$& 32257 & 58490 &103791 &180603 &308780 &519629 &861849 & 1410525 &2280428 \\ \hline
$ch_{q}[M_{5}]ch_{q}[L_{\text{Vir}}(c_{10,7},h_{10,7}^{5,1})]$& 526 & 1235 &2739 &5760 &11625 &22656 &42847 & 78912 &142047\\ \hline
$ ch_{q}[L_{\widehat{osp(1|2)}}(1, 0) ^{\otimes 2}] $& 75940 &140219 &253150 & 447857 &777940 & 1329196 & 2236966 & 3712731 & 6083848 \\ \hline

Coefficients of items in (\ref{4e})&  75940 &140219 &253150 & 447857 &777940 & 1329196 & 2236966 & 3712731 & 6083848 \\ \hline
\end{tabular}
}
\end{table}

\noindent
The lemma immediately follows from Table 2.
\end{proof}
\begin{theorem}\label{thm11111}	
(1)	  $C( L_{\widehat{osp(1|2)}}(2,0),  L_{\widehat{osp(1|2)}}(1,0) ^{\otimes 2})$ is a conformal extension of the simple Virasoro vertex operator algebra $L(c_{10,7},0)$.
	
(2) 
   \begin{equation}\label{ethm1}
   \begin{aligned}
L_{\widehat{osp(1|2)}}(1,0) ^{\otimes 2} =  & L_{\widehat{osp(1|2)}}(2,0) \otimes (L_{\text{Vir}}(c_{10,7},0) \oplus L_{\text{Vir}}(c_{10,7},h_{10,7}^{6,1})) \\
   & + M_{3} \otimes (L_{\text{Vir}}(c_{10,7},h_{10,7}^{3,1}) \oplus L_{\text{Vir}}(c_{10,7},h_{10,7}^{4,1}))\\
   & + M_{5} \otimes (L_{\text{Vir}}(c_{10,7},h_{10,7}^{2,1}) \oplus L_{\text{Vir}}(c_{10,7},h_{10,7}^{5,1})).
   \end{aligned}
   \end{equation}
  
  (3) $C( L_{\widehat{osp(1|2)}}(2,0),  L_{\widehat{osp(1|2)}}(1,0) ^{\otimes 2})=L_{\text{Vir}}(c_{10,7},0) \oplus L_{\text{Vir}}(c_{10,7},h_{10,7}^{6,1}))$ is a vertex operator algebra. 
\end{theorem}

\begin{proof} By \cite{CL},  $L_{\widehat{osp(1|2)}}(2, 0)$ is rational. So $L_{\widehat{osp(1|2)}}(1, 0) ^{\otimes 2}$ is completely reducible as a module of $L_{\widehat{osp(1|2)}}(2, 0)$.  By Theorem \ref{thmCFK}, $L_{\widehat{osp(1|2)}}(2, 0)$ has three irreducible modules: $M_1=L_{\widehat{osp(1|2)}}(2, 0)$, $M_3$,  and $M_5$. Then we have 
\begin{equation}\label{decomposition}	
	L_{\widehat{osp(1|2)}}(1, 0) ^{\otimes 2}=L_{\widehat{osp(1|2)}}(2, 0)\otimes U\oplus M_3\otimes W^3\oplus M_5\otimes W^5,
\end{equation}	
	where 
	$U=C( L_{\widehat{osp(1|2)}}(2,0), L_{\widehat{osp(1|2)}}(1,0) ^{\otimes 2})$, and $W^3$ and $W^5$ are modules of  $U$.
	
Notice that the conformal weight of $L(l,i) \otimes V_{i+1,r}$ is
\begin{align}
  \frac{1}{4}(2(i+1)^2-2(i+1)r+\frac{(l+2)(r^2-1)}{2l+3}).
\end{align}
Using  (\ref{Meven}) and (\ref{Modd}), one can  check directly that the conformal weight of $M_{3}$ is $\frac{1}{7}$ and the lowest weight space $(M_3)_0$ of $M_3$ is 3-dimensional, while the conformal weight of $M_{5}$ is $\frac{3}{7}$ and the lowest weight space $(M_5)_0$ of $M_5$ is 5-dimensional.  Denote by $\omega'$ the conformal vector of $U$ and $X$ the Virasoro vertex subalgebra generated by $\omega'.$  By Lemma \ref{cc}, the central charge of $\omega'$ is $c_{10,7}$.  So $X$ is a quotient of the universal Virasoro vertex operator algebra $V(c_{10,7},0)$. 
It can be checked directly that there exist  primary vectors  $u^1\in W_5$ and $u^2\in W_3$ of  the Virasoro vertex operator algebra $X$,  that is, 
$$
\omega'_1u^1=\frac{4}{7}u^1=h_{10,7}^{2,1}u^1, \  ~  \omega'_{m}u^1=0, \ ~ m\geq 2, 
$$
and
$$
\omega'_1u^1=\frac{13}{7}u^1=h_{10,7}^{3,1}u^2, \  ~  \omega'_{m}u^2=0, \ ~ m\geq 2.
$$
This together with Table 2  deduces that the lowest weight subspaces of $W_5$ and $W_3$ are $\mathbb {C}u^1$ and ${\mathbb C}u^2$ respectively, and  $U$ is a conformal extension of $X$.  Denote by $Y^1$ the $X$-submodule of $W_5$ generated by $u^1$, and $Y^2$ the $X$-submodule of $W_3$ generated by $u^2$.  For non-zero  $t\in{\mathbb C}$,  and $\alpha, \beta\in{\mathbb Z}$, denote
\begin{equation}\label{ee1}
h_{\alpha,\beta}(t)=\frac{1}{4}(\alpha^2-1)t-\frac{1}{2}(\alpha\beta-1)+\frac{1}{4}(\beta^2-1)t^{-1}.
\end{equation}
Notice that $h_{r,s}(\frac{q}{p})=h_{p,q}^{r,s}$. Recall from \cite{FF},  \cite{IK} that the maximal submodule of the Verma  module $M(c_{p,q}, h_{r,s}(\frac{p}{q}))$ of the universal Virasoro Lie algebra is generated by two singular vectors of weights $h_{r,-s}(\frac{p}{q})$ and $h_{-2q+r,-s}(\frac{p}{q})$, respectively.  Then if $Y^1$ is not simple, $Y^1$ has at least one singular vector  of weight  $h_{2, -1}(\frac{10}{7})=\frac{4}{7}+2$ or 
$h_{-12, -1}(\frac{10}{7})=\frac{4}{7}+45$. Similarly,  if $Y^2$ is not simple, $Y^2$ has at least one singular vector  of weight  $h_{3, -1}(\frac{10}{7})=\frac{13}{7}+3$ or 
$h_{-11, -1}(\frac{10}{7})=\frac{13}{7}+36$.  Then by Table 2, there is no singular vector of weight $\frac{4}{7}+2$ in $Y^1$, while there is a primary vector $u^3\in W_3$ of $X$ of weight $h_{10,7}^{4,1}=\frac{13}{7}+2$.  This together with Table 2  implies that  there is no singular vector in $Y^2$ of weight $\frac{13}{7}+3$.  Denote by $Y^3$ the $X$-submodule of $W_3$ generated by $u^3$. If $Y^3$ is not simple, then $Y^3$ has at least one singular vector of weight $h_{4,-1}(\frac{10}{7})=\frac{13}{7}+6$ or $h_{-10,-1}(\frac{10}{7})=\frac{13}{7}+29$. Again by Table 2, there exists a primary vector $u^4$ of $X$  in $W_5$  of weight $h_{10,7}^{5,1}=\frac{4}{7}+6$ and there is no singular vector in $Y^3$ of weight $\frac{13}{7}+6$. Denote by $Y^4$ the $X$-submodule of $W_5$ generated by $u^4$.  If $Y^4$ is not simple, then there exists at least one singular vector in $Y^4$ of weight $h_{5,-1}(\frac{10}{7})=\frac{4}{7}+11$ or $h_{-9, -1}(\frac{10}{7})=\frac{4}{7}+25$. 
Then by Table 2, there is a primary vector $u^5$ in $U$ of weight $h_{10,7}^{6,1}=10$, while there is no singular vector in $Y^4$ of weight $\frac{4}{7}+11$. Denote by $Y^5$ the $X$-submodule of $U$ generated by $u^5$. If $Y^5$ is not simple, then there is at least one singular vector in $Y^5$ of weight $h_{6,-1}(\frac{10}{7})=16$ or $h_{-8,-1}(\frac{10}{7})=19$. But  by Table 2, there is no such singular vector in $Y^5$. This means that 
$$
Y^5\cong L(c_{10,7}, h_{10,7}^{6,1}).
$$
Notice that if
   $X$ is not simple, then $X$ has a singular vector $x$ of weight $54$.  By \cite{FZ} (also see \cite{L}),  the structure of $ L(c_{10,7}, h_{10,7}^{2,1})$ as a module of the universal Virasoro vertex operator algebra $V(c_{10,7},0)$ is unique.  Then 
$$
Y(x,z)u^5=0.
$$This contracts Lemma 3.1 of \cite{DM}, since $L_{\widehat{osp(1|2)}}(1,0) ^{\otimes 2}$ is simple. We deduce that 
$$X\cong L(c_{10,7},0).$$
Hence $U$ is a conformal extension of $L(c_{10,7},0)$.  

\vskip0.2cm
Since $L(c_{10,7},0)$ is rational,  as $L(c_{10,7},0)$-modules, $U$, $W_3$, and $W_5$ are direct sums of irreducible modules of $L(c_{10,7},0)$, respectively, and 
$$
Y^1\cong L(c_{10,7},h_{10,7}^{2,1}), \  Y^2\cong L(c_{10,7},h_{10,7}^{3,1}), \  Y^3\cong L(c_{10,7},h_{10,7}^{4,1}), \ Y^4\cong L(c_{10,7},h_{10,7}^{5,1}).
$$
To achieve an integral weight summation,  $M_{3}$ can only be paired with $L_{\text{Vir}}(c_{10,7}, h_{10,7}^{3,1})$ and $L_{\text{Vir}}(c_{10,7}, h_{10,7}^{4,1})$, whereas $M_{5}$ can only be paired with $L_{\text{Vir}}(c_{10,7}, h_{10,7}^{2,1})$ and $L_{\text{Vir}}(c_{10,7}, h_{10,7}^{5,1})$. Then  (\ref{ethm1}) follows and 
\begin{align}
U=	C( L_{\widehat{osp(1|2)}}(2,0),L_{\widehat{osp(1|2)}}(1,0) ^{\otimes 2} ) =  L_{\text{Vir}}(c_{10,7},0) \oplus L_{\text{Vir}}(c_{10,7},h_{10,7}^{6,1}).
\end{align}

We now prove that $L_{\text{Vir}}(c_{10,7},0) \oplus L_{\text{Vir}}(c_{10,7},h_{10,7}^{6,1})$ is a vertex operator algebra.
By (\ref{Meven})-(\ref{Modd}), 
 we have the following $q$-characters.
\begin{align}
ch_{q}[L_{\widehat{osp(1|2)}}^{even}(l, 0)] = \sum_{i = 0, i \text{ even}} ^{l}ch_{q}[L(l,i)]ch_{q}[L_{\text{Vir}}(c_{2l+3,l+2},h_{2l+3,l+2}^{i+1,1})],
\end{align}
\begin{align}
ch_{q}[L_{\widehat{osp(1|2)}}^{odd}(l, 0)] = \sum_{i = 0, i \text{ odd}} ^{l}ch_{q}[L(l,i)]ch_{q}[L_{\text{Vir}}(c_{2l+3,l+2},h_{2l+3,l+2}^{i+1,1})],
\end{align} 
\begin{align}
ch_{q}[M^{even}_{r}] = \sum_{i = 0, i \text{ even}}^{l}ch_{q}[L(l,i)]ch_{q}[L_{\text{Vir}}(c_{7,4},h_{7,4}^{i+1,r})],
\end{align}
\begin{align}
ch_{q}[M^{odd}_{r}] = \sum_{i = 0, i \text{ odd}}^{l}ch_{q}[L(l,i)]ch_{q}[L_{\text{Vir}}(c_{7,4},h_{7,4}^{i+1,r})],
\end{align}
where $r =3$ or $r = 5$.
By Theorems \ref{the1}-\ref{cha}, we can deduce that
\begin{align*}
& ch_{q}[L^{even}_{\widehat{osp(1|2)}}(1, 0) ^{\otimes 2}]  =  q^{-\frac{1}{30}}(1 + 6q + 23q^2 +68q^3+191q^4+478q^5 +1107q^6 \\  & + 2436q^7+5108q^8+10290q^9+20068q^{10}+\cdots), \\
& ch_{q}[L_{\widehat{osp(1|2)}}^{even}(2, 0)]ch_{q}[L_{\text{Vir}}(c_{10,7},0)]   =  q^{-\frac{1}{30}}(1 + 3q + 11q^2 +26q^3+66q^4+148q^5 \\  &+317q^6 + 648q^7+1281q^8+2438q^9+4533q^{10}+\cdots), \\
& ch_{q}[L_{\widehat{osp(1|2)}}^{even}(2, 0)]ch_{q}[L_{\text{Vir}}(c_{10,7},h_{10,7}^{6,1})]   =  q^{-\frac{1}{30}}(q^{10}+\cdots), \\
&  ch_{q}[M^{even}_{3}]ch_{q}[L_{\text{Vir}}(c_{10,7},h_{10,7}^{3,1})]   =  q^{-\frac{1}{30}}(q^2 +8q^3+30q^4+91q^5 +237q^6 + 567q^7 \\ & +1263q^8+2670q^9+5397q^{10}+\cdots), \\
& ch_{q}[M^{even}_{3}]ch_{q}[L_{\text{Vir}}(c_{10,7},h_{10,7}^{4,1})]   =  q^{-\frac{1}{30}}(q^4+8q^5+30q^6 + 92q^7+244q^8+589q^9 \\ & +1325q^{10}+\cdots), \\
&  ch_{q}[M^{even}_{5}]ch_{q}[L_{\text{Vir}}(c_{10,7},h_{10,7}^{2,1})]  =  q^{-\frac{1}{30}}(3q+11q^2 +34q^3+94q^4+231q^5 +523q^6 \\ & + 1126q^7+2309q^8+4556q^9+8707q^{10}+\cdots), \\
&  ch_{q}[M^{even}_{5}]ch_{q}[L_{\text{Vir}}(c_{10,7},h_{10,7}^{5,1})]  =  q^{-\frac{1}{30}}(3q^7+11q^8+37q^9+105q^{10}+\cdots), \\
\end{align*}
\begin{align*}
 & ch_{q}[L^{odd}_{\widehat{osp(1|2)}}(1, 0) ^{\otimes 2}]   =  q^{-\frac{1}{30}}(4q + 20q^2 +64q^3+184q^4+468q^5 +1092q^6 + 2416q^7  \\ &+5080q^8+10252q^9+20016q^{10}+\cdots), \\
 & ch_{q}[L_{\widehat{osp(1|2)}}^{odd}(2, 0)]ch_{q}[L_{\text{Vir}}(c_{10,7},0)]  =  q^{-\frac{1}{30}}(2q + 8q^2 +22q^3+58q^4+136q^5+296q^6  \\  & + 618q^7+1232q^8+2368q^9+4426q^{10}+\cdots), \\
 & ch_{q}[L_{\widehat{osp(1|2)}}^{odd}(2, 0)]ch_{q}[L_{\text{Vir}}(c_{10,7},h_{10,7}^{6,1})]   =  q^{-\frac{1}{30}}(0q^{10}+\cdots), \\
 & ch_{q}[M^{odd}_{3}]ch_{q}[L_{\text{Vir}}(c_{10,7},h_{10,7}^{3,1})]   =  q^{-\frac{1}{30}}(2q^2 +10q^3+34q^4+98q^5 +250q^6 + 588q^7 \\  &+1298q^8+2724q^9+5482q^{10}+\cdots), \\
 &  ch_{q}[M^{odd}_{3}]ch_{q}[L_{\text{Vir}}(c_{10,7},h_{10,7}^{4,1})]   =  q^{-\frac{1}{30}}(2q^4+10q^5+34q^6 + 100q^7+258q^8 +612q^9 \\ & +1364q^{10}+\cdots), \\
 &  ch_{q}[M^{odd}_{5}]ch_{q}[L_{\text{Vir}}(c_{10,7},h_{10,7}^{2,1})]   =  q^{-\frac{1}{30}}(2q+10q^2 +32q^3+90q^4+224q^5+512q^6 \\  &+ 1108q^7+2282q^8+4514q^9+8644q^{10}+\cdots), \\ 
 &  ch_{q}[M^{odd}_{5}]ch_{q}[L_{\text{Vir}}(c_{10,7},h_{10,7}^{5,1})]  =  q^{-\frac{1}{30}}(2q^7+10q^8+34q^9+100q^{10}+\cdots).
\end{align*}

By comparing the coefficients of $q$, $q^2$, and $q^3$, we deduce that 
$M^{even}_{3}\otimes L_{\text{Vir}}(c_{10,7},h_{10,7}^{3,1})$ and $M^{even}_{5}\otimes L_{\text{Vir}}(c_{10,7},h_{10,7}^{2,1})$ are included in  $L^{even}_{\widehat{osp(1|2)}}(1, 0) ^{\otimes 2}$. 
Comparing the coefficients of each power of $q^4$, we see that   $M^{even}_{3}\otimes L_{\text{Vir}}(c_{10,7},h_{10,7}^{4,1})$  is also included in $L^{even}_{\widehat{osp(1|2)}}(1, 0) ^{\otimes 2}$.  By the fusion rules of irreducible modules of $ L_{\text{Vir}}(c_{10,7}, 0)$, $ L_{\text{Vir}}(c_{10,7}, h_{10,7}^{6,1})$ is in the even part of $L{\widehat{osp(1|2)}}(1, 0) ^{\otimes 2}$. It follows that 
$L_{\text{Vir}}(c_{10,7},0) \oplus L_{\text{Vir}}(c_{10,7}, h_{10,7}^{6,1})$ is a vertex operator subalgebra of $L_{\widehat{osp(1|2)}}(1,0)^{\otimes 2}$. 
\end{proof}
\section{Irreducible modules and fusion rules of \\ $C( L_{\widehat{osp(1|2)}}(2,0), L_{\widehat{osp(1|2)}}(1,0) ^{\otimes 2})$ }

\subsection{Irreducible modules of $C( L_{\widehat{osp(1|2)}}(2,0),L_{\widehat{osp(1|2)}}(1,0) ^{\otimes 2})$}

In this section, we determine the irreducible modules  and fusion rules  of  $C( L_{\widehat{osp(1|2)}}(2,0),$ $ L_{\widehat{osp(1|2)}}(1,0) ^{\otimes 2})$.
 Recall the following results  from \cite{DL2}, \cite{JW},  and \cite{Lam}.

\begin{theorem}\label{thm4}\cite{DL2}\cite{JW}\cite{Lam}
Let $V$ be a simple rational and $C_{2}$-cofinite vertex operator algebra. Let $M$ be a simple current module of $V$ such that $\widetilde{V} = V \oplus M$ possesses a vertex operator algebra structure. Then any irreducible $\widetilde{V}$-module, regarded as a $V$-module, is of the following form:
\begin{itemize}
 \item $N \oplus \widetilde{N}$, where $N$ is an irreducible $V$-module and $\widetilde{N} = M \boxtimes_{V}N \ncong N$. In this case, $N \oplus \widetilde{N}$ has a unique $\widetilde{V}$-module structure.
 \item $N$, where $N$ is an irreducible $V$-module and $N \cong M \boxtimes_{V} N$. In this case, there exist two non-isomorphic  $\widetilde{V}$-module structures on $N$.
\end{itemize}

\end{theorem}

\begin{theorem}\label{thmf}
All irreducible modules of the commutant vertex operator algebra  $C( L_{\widehat{osp(1|2)}}(2,0), $ $L_{\widehat{osp(1|2)}}(1,0) ^{\otimes 2})$ are given by  $L_{\text{Vir}}(c_{10,7},h_{10,7}^{r,s}) \oplus L_{\text{Vir}}(c_{10,7},h_{10,7}^{7-r,s})$ 
where 
\begin{align*}
  h_{10, 7}^{r, s}=\frac{(10r-7s )^2-9}{280}, 1 \leq r \leq 3,1 \leq s \leq 9.
\end{align*}
\end{theorem}

\begin{proof}
From (\ref{6}),    \( L_{\text{Vir}}(c_{10,7},h_{10,7}^{r,s}) \),  \( 1 \leq r \leq 6 \) and \( 1 \leq s \leq 9 \) exhaust all the irreducible modules of \( L_{\text{Vir}}(c_{10,7},0) \). By Theorem \ref{wang}, we have
\begin{equation}\label{vvvv}
\begin{aligned}
L_{\text{Vir}}(c_{10,7},h_{10,7}^{r,s}) \boxtimes_{\text{Vir}} L_{\text{Vir}}(c_{10,7},h_{10,7}^{6,1}) \cong \bigoplus_{r^{\prime }=1}^{6} \bigoplus_{s^{\prime}=1}^{9} N_{(r,s),(6,1)}^{(r',s')} L_{\text{Vir}}(c_{10.7},h_{10,7}^{r',s'})
\end{aligned}
\end{equation}
and $ N_{(r,s),(6,1)}^{(r',s')} = 1 $ if and only if $ ((r,s),(6,1),(r',s')) $ forms an admissible triple. Thus,  $ N_{(r,s),(6,1)}^{(r',s')} = 1 $ if and only if $ r + r' = 7 $ and $ s = s' $. Therefore, we conclude that $ r \neq r' $ and 
\begin{align}
L_{\text{Vir}}(c_{10,7},h_{10,7}^{r,s}) \boxtimes_{\text{Vir}} L_{\text{Vir}}(c_{10,7},h_{10,7}^{6,1}) \cong L_{\text{Vir}}(c_{10,7},h_{10,7}^{7-r,s}) \ncong L_{\text{Vir}}(c_{10,7},h_{10,7}^{r,s}).
\end{align}

By Theorem \ref{thm4},  all irreducible modules of $L_{\text{Vir}}(c_{10,7},0) \oplus L_{\text{Vir}}(c_{10,7},h_{10,7}^{6,1})$ are given   by $L_{\text{Vir}}(c_{10,7},h_{10,7}^{r,s}) \oplus L_{\text{Vir}}(c_{10,7},h_{10,7}^{7-r,s})$, $1\leq r\leq 3$, $1\leq s\leq 9$.

By Theorem  \ref{thm11111}, $C_{L_{\widehat{osp(1|2)}}(1,0) ^{\otimes 2} }( L_{\widehat{osp(1|2)}}(2,0))=L_{\text{Vir}}(c_{10,7},h_{10,7}^{r,s}) \oplus L_{\text{Vir}}(c_{10,7},h_{10,7}^{7-r,s})$.  Thus the  theorem holds.
\end{proof}

\subsection{Fusion rules of $C( L_{\widehat{osp(1|2)}}(2,0), L_{\widehat{osp(1|2)}}(1,0) ^{\otimes 2})$ }
Assume that $ V $ is a simple rational and $ C_2 $-cofinite vertex operator algebra, and let $ W $ be a simple current module of $ V $ such that $ \widetilde{V} = V \oplus W $ possesses a vertex operator algebra structure. Let $ M = M_{0} \oplus M_{1} $ and $ N = N_{0} \oplus N_{1} $ be two $ \widetilde{V} $-modules. From \cite{CKM}, \cite{KO}, \cite{Li}, we have
\begin{align}\label{e4.1}
    M \boxtimes_{\widetilde{V}} N & = (\widetilde{V} \boxtimes M_{0}) 
    \boxtimes_{\widetilde{V}} (\widetilde{V}  \boxtimes N_{0})= \widetilde{V} \boxtimes(M_{0} \boxtimes N_{0} ).
\end{align}

 For simplicity, we denote 
 $$\widetilde{V}=C( L_{\widehat{osp(1|2)}}(2,0), L_{\widehat{osp(1|2)}}(1,0) ^{\otimes 2})=L_{\text{Vir}}(c_{10,7},0) \oplus L_{\text{Vir}}(c_{10,7},h_{10,7}^{6,1}),$$
 and
 $$\widetilde{V}_{r,s}=L_{\text{Vir}}(c_{10,7},h_{10,7}^{r,s}) \oplus L_{\text{Vir}}(c_{10,7},h_{10,7}^{7-r,s}),$$
 where  $r=1,3,5$, $1\leq s\leq 9$.
 
   The fusion rules for irreducible modules of $\widetilde{V}$ are given in the theorem below.
\begin{theorem}
 For $1\leq r,r'\leq 3$, $1\leq s, s'\leq 9$, we have
\begin{equation} \widetilde{V}_{r,s}\boxtimes_{\widetilde{V}}\widetilde{V}_{r',s'}  \\
=\bigoplus_{r^{\prime \prime}=1,odd}^{5} \bigoplus_{s^{\prime \prime}=1}^{9}N_{(r,s),(r',s')}^{(r'',s'')}\widetilde{V}_{r'',s''},
\end{equation}
where $ N_{(r,s),(r',s')}^{(r'',s'')} =1 $ if and only if  $((r,s),(r',s'),(r'',s''))$ is an admissible triple.
\end{theorem}

\begin{proof} It is enough to prove the theorem for the cases that $r=r'=3$, $r=r'=5$, and $r=3$, $r'=5$.   
	By (\ref{e4.1}), 
	\begin{align*}& \widetilde{V}_{5,s}\boxtimes_{\widetilde{V}}\widetilde{V}_{5,s'}\\
		=& (L_{\text{Vir}}(c_{10,7},h_{10,7}^{2,s}) \oplus L_{\text{Vir}}(c_{10,7},h_{10,7}^{5,s})) \boxtimes_{\widetilde{V}} (L_{\text{Vir}}(c_{10,7},h_{10,7}^{2,s'}) \oplus L_{\text{Vir}}(c_{10,7},h_{10,7}^{5,s'})) \\
		=&  (L_{\text{Vir}}(c_{10,7},0) \oplus L_{\text{Vir}}(c_{10,7},h_{10,7}^{6,1}))  \boxtimes ( L_{\text{Vir}}(c_{10,7},h_{10,7}^{2,s}) \boxtimes  L_{\text{Vir}}(c_{10,7},h_{10,7}^{2,s'})) \\
		=&  (L_{\text{Vir}}(c_{10,7},0) \oplus L_{\text{Vir}}(c_{10,7},h_{10,7}^{6,1}))  \boxtimes[\bigoplus_{s''=1}^{9}  N_{s,s'}^{s''} (L_{\text{Vir}}(c_{10,7},h_{10,7}^{1,s''}) \oplus L_{\text{Vir}}(c_{10,7},h_{10,7}^{3,s''}))]\\
	=	& (\bigoplus_{s^{\prime \prime}=1}^{9}N_{s,s'}^{s''} \widetilde{V}_{1,s''} ) \bigoplus (\bigoplus_{s^{\prime \prime}=1}^{9}N_{s,s'}^{s''} \widetilde{V}_{3,s''})=\bigoplus_{r^{\prime \prime}=1,odd}^{5} (\bigoplus_{s^{\prime \prime}=1}^{9}N_{(5,s),(5,s')}^{(r'',s'')}\widetilde{V}_{r'',s''}),
	\end{align*}
	where $N_{s,s'}^{s''} =1$  if $\{s, s', s''\}$ is an admissible triple, i.e., $1\leq s,s',s''\leq 9$,  $s''<s+s'$, $s'<s+s''$, $s<s'+s''$, $s+s'+s''$ is odd, and $s+s'+s''\leq 19$.  Otherwise,  $N_{s,s'}^{s''} =0$.  Similarly, 
	\begin{align*}& \widetilde{V}_{3,s}\boxtimes_{\widetilde{V}}\widetilde{V}_{3,s'}\\
	=& (L_{\text{Vir}}(c_{10,7},h_{10,7}^{3,s}) \oplus L_{\text{Vir}}(c_{10,7},h_{10,7}^{4,s})) \boxtimes_{\widetilde{V}} (L_{\text{Vir}}(c_{10,7},h_{10,7}^{3,s'}) \oplus L_{\text{Vir}}(c_{10,7},h_{10,7}^{4,s'})) \\
	=& (L_{\text{Vir}}(c_{10,7},0) \oplus L_{\text{Vir}}(c_{10,7},h_{10,7}^{6,1}))  \boxtimes ( L_{\text{Vir}}(c_{10,7},h_{10,7}^{3,s}) \boxtimes  L_{\text{Vir}}(c_{10,7},h_{10,7}^{3,s'})) \\
	=& (L_{\text{Vir}}(c_{10,7},0) \oplus L_{\text{Vir}}(c_{10,7},h_{10,7}^{6,1}))  \boxtimes[\bigoplus_{s''=1}^{9}  N_{s,s'}^{s''} (L_{\text{Vir}}(c_{10,7},h_{10,7}^{1,s''})\\
	& \oplus L_{\text{Vir}}(c_{10,7},h_{10,7}^{3,s''}) \oplus L_{\text{Vir}}(c_{10,7},h_{10,7}^{5,s''}))]\\
	=& (\bigoplus_{s^{\prime \prime}=1}^{9}N_{s,s'}^{s''} \widetilde{V}_{1,s''} ) \bigoplus (\bigoplus_{s^{\prime \prime}=1}^{9}N_{s,s'}^{s''} \widetilde{V}_{3,s''})\bigoplus (\bigoplus_{s^{\prime \prime}=1}^{9}N_{s,s'}^{s''} \widetilde{V}_{5,s''})\\
=&\bigoplus_{r^{\prime \prime}=1,odd}^{5} (\bigoplus_{s^{\prime \prime}=1}^{9}N_{(3,s),(3,s')}^{(r'',s'')}\widetilde{V}_{r'',s''}),
\end{align*}
and
	\begin{align*}& \widetilde{V}_{3,s}\boxtimes_{\widetilde{V}}\widetilde{V}_{5,s'}\\
	=& (L_{\text{Vir}}(c_{10,7},h_{10,7}^{3,s}) \oplus L_{\text{Vir}}(c_{10,7},h_{10,7}^{4,s})) \boxtimes_{\widetilde{V}} (L_{\text{Vir}}(c_{10,7},h_{10,7}^{5,s'}) \oplus L_{\text{Vir}}(c_{10,7},h_{10,7}^{2,s'})) \\
	=& (L_{\text{Vir}}(c_{10,7},0) \oplus L_{\text{Vir}}(c_{10,7},h_{10,7}^{6,1}))  \boxtimes ( L_{\text{Vir}}(c_{10,7},h_{10,7}^{3,s}) \boxtimes  L_{\text{Vir}}(c_{10,7},h_{10,7}^{5,s'})) \\
	=& (L_{\text{Vir}}(c_{10,7},0) \oplus L_{\text{Vir}}(c_{10,7},h_{10,7}^{6,1}))  \boxtimes[\bigoplus_{s''=1}^{9}  N_{s,s'}^{s''} (L_{\text{Vir}}(c_{10,7},h_{10,7}^{3,s''}) \oplus L_{\text{Vir}}(c_{10,7},h_{10,7}^{5,s''})]\\
	=& (\bigoplus_{s^{\prime \prime}=1}^{9}N_{s,s'}^{s''} \widetilde{V}_{3,s''} ) \bigoplus (\bigoplus_{s^{\prime \prime}=1}^{9}N_{s,s'}^{s''} \widetilde{V}_{5,s''})\\
	=&\bigoplus_{r^{\prime \prime}=1,odd}^{5}( \bigoplus_{s^{\prime \prime}=1}^{9}N_{(3,s),(5,s')}^{(r'',s'')}\widetilde{V}_{r'',s''}).
\end{align*}

\end{proof}

\end{document}